\documentclass[runningheads]{llncs}
\usepackage{graphicx}
\usepackage{amsmath}

\begin{document}
\title{Some results on domination number of the graph defined by two levels of the $n$-cube}
%
%
\author{Yeshwant Pandit\inst{1}\and
S.L.Sravanthi\inst{1} \and
Suresh Dara\inst{2}\and
S.M.Hegde\inst{1}
}

\institute{Department of Mathematical and Computational Sciences,\\National Institute of Technology Karnataka, Surathkal, Mangalore, India\\
\email{ypandit2001@gmail.com; sravanthi.settaluri@gmail.com; smhegde@nitk.ac.in}\and 
Department of Mathematics,
Birla Institute of Technology,\\
Mesra, Ranchi, India\\
\email{suresh.dara@gmail.com}
}
\maketitle              
\begin{abstract}
Let ${[n] \choose k}$ and ${[n] \choose l}$ $( k > l ) $ where $[n] = \{1,2,3,...,n\}$ denote the family of all $k$-element subsets and $l$-element subsets of $[n]$ respectively. Define a bipartite graph $G_{k,l} = ({[n] \choose k},{[n] \choose l},E)$ such that two vertices $S\, \epsilon \,{[n] \choose k} $ and $T\, \epsilon \,{[n] \choose l} $ are adjacent if and only if $T \subset S$. In this paper, we give an upper bound for the domination number of graph $G_{k,2}$ for $k > \lceil \frac{n}{2} \rceil$ and exact value for $k=n-1$.

\end{abstract}
\section{Introduction}

Let $G=(V,E)$ be a simple graph, a set $S \subseteq V$ is a dominating set if every vertex not in $S$ has a neighbour in $S$. The minimum size of a dominating set in $G$ is called as the domination number of $G$ and is denoted by $\gamma(G)$. 


Let $G_{k,l}$ denote the bipartite graph $G = ({[n] \choose k},{[n] \choose l},E)$ where $[n] = \{1,2,3,...,n\}$ for $n > k > l \geq 1$; ${[n] \choose k}$ and ${[n] \choose l}$ denote the family of all \textit{k}-element subsets and all \textit{l}-element subsets of $[n]$ respectively. Two vertices $S\, \epsilon \,{[n] \choose k} $ and $T\, \epsilon \,{[n] \choose l} $ are adjacent iff $T \subset S$. As mentioned in \cite{1}, the family $ {[n] \choose k} $ is called the $k^{th}$ level of the $n$-cube and $G_{k,l}$ is called as the graph defined by the $k^{th}$ level and the $l^{th}$ level. 

In \cite{1}, Badakhshian, Katona and Tuza proved that  $\gamma(G_{k,1})= n-k+1$ for $k\geq 2$. They have also given upper and lower bounds for $\gamma(G_{k,2})$ and posed the following conjecture.

\begin{conjecture}\label{c1}
	 $\gamma(G_{k,2}) = \frac{k+3}{2(k-1)(k+1)}n^{2} + o(n^{2})$ for $k\geq 3$.
	\end{conjecture}

\section{Results}
In this section, we give an upper bound for $\gamma(G_{k,2})$ for $k > \lceil \frac{n}{2} \rceil$ and exact value of  $\gamma(G_{k,2})$ for $k=n-1$.

\begin{theorem}\label{t1} $\gamma(G_{k,2}) \leq \lceil \frac{n}{2} \rceil + 6$ for $k > \lceil \frac{n}{2} \rceil$.
\end{theorem}
\begin{proof}
	Let $S$ and $T$ be the two $k$-element subsets of $[n]$ such that $S \cup T = [n]$.
	
	Let $k$ be even, then define the set $A$ as $$A=\{S,T,P_1,P_2,P_3,P_4\},$$ where $P_1,P_2,P_3,P_4$ be the $k$-element subsets of $[n]$, which are defined as follows:
	 $$P_1 = S_1 \cup T_1, P_2 = S_1 \cup T_2, P_3 = S_2 \cup T_1, P_4 = S_2 \cup T_2,$$
where $S_1, S_2$ are subsets of $S$ such that $|S_1|=|S_2|=\frac{k}{2}$ and $S_1 \cup S_2 = S$, and $T_1, T_2$ are subsets of $T$ such that $|T_1|=|T_2|=\frac{k}{2}$ and $T_1 \cup T_2 = T$.\\

Let $k$ be odd. Since $k > \lceil \frac{n}{2} \rceil$, there exist an element $l$(say) which belongs to $S\cap T$. Define the set $A$ as $$A=\{S,T,P_1,P_2,P_3,P_4\},$$ where $P_1,P_2,P_3,P_4$ be the $k$-element subsets of $[n]$, which are defined as follows:
$$P_1 = S_1 \cup T_1, P_2 = S_1 \cup T_2, P_3 = S_2 \cup T_1, P_4 = S_2 \cup T_2,$$
where $S_1, S_2$ are subsets of $S$ such that $|S_1|=|S_2|=\frac{k-1}{2}$ and $S_1 \cup S_2 = S\setminus \{l\}$, and $T_1, T_2$ are subsets of $T$ such that $|T_1|=|T_2|=\frac{k+1}{2}$ and $T_1 \cup T_2 = T$ and $T_1 \cap T_2 = \{l\}$. Also, one can note that in both the cases ($k$ is odd or even) if $|P_i| < k$, then it can be made $k$ by adding elements from $[n]$.\\

 Let $\{a,b\}$ be any $2$-element subset of $[n]$, then $a,b \in S\cup T$. If, both $a,b$ are the elements of $S$ or $T$, then the set $\{a,b\}$ is dominated by $S$ or $T$ respectively. If not, without loss of generality, it follows that $a \in S$ and $b\in T$. By the definition of $S_1, S_2, T_1$ and $T_2$, $a$ belongs to $S_1$ or $S_2$ and $b$ belongs to $T_1$ or $T_2$. This implies that $a,b$ belongs to $P_i$ for some $i$, $i = 1,2,3,4$. Hence, $A$ dominates every $2$-element subset of $[n]$.

Let $B$ be the set of $\lceil \frac{n}{2} \rceil$ $2$-element subsets of $[n]$ such that $B$ spans $[n]$. Let $S$ be any $k$-element subset of $[n]$, then $|S| \geq \lceil \frac{n}{2} \rceil +1$. Suppose that, no element of $B$ is adjacent to $S$. Then for every element $\{x,y\}$ of $B$, $S$ contains either $x$ or $y$ but not both, implies that $|S| \leq  \lceil \frac{n}{2} \rceil$, which is a contradiction. Hence, $B$ dominates every $k$-element subset of $[n]$.

Therefore, $\gamma(G_{k,2}) \leq |B| + |A| \leq \lceil \frac{n}{2} \rceil + 6 $ for $k> \lceil \frac{n}{2} \rceil$. Hence, the proof.
\end{proof}

\begin{theorem}\label{t2}
	$\gamma(G_{n-1,2}) =3$. 
\end{theorem}

\begin{proof}
Let $P_1=\{1,2,3, \dots n-1\}$, $P_2= \{2, 3, 4, \dots, n\}$, $P_3=\{1,n\}$ and $D =\{P_1, P_2, P_3\}$. Now, we show that $D$ is a dominating set of $G_{n-1,2}$. Let $S$ be any vertex of the graph $G_{n-1,2}$, then $S$ is either an $(n-1)$-element set or a $2$-element set. 

Let $S$ be an $(n-1)$-element set, then $S$ will be of the form $[n]\setminus \{i\}$. If $S \in D$ we are done, if not $S$ contains $1,n$, therefore $S$ will be dominated by $P_3$.

Let $S$ be a $2$-element set, then $S$ will be of the form of $\{a,b\}$. If $S$ is different from $P_3$, then $S$ will be dominated either by $P_1$ or by $P_2$.

Hence, $D$ is a dominating set of $G_{n-1,2}$. Therefore, $\gamma(G_{n-1,2}) \leq 3$

Now, we prove that $\gamma(G_{n-1,2}) >2$. Since $G_{n-1,2}$ is a bipartite graph, $\gamma(G_{n-1,2}) \geq 2$. Suppose that $\gamma(G_{n-1,2}) = 2$, then the dominating set contains one $(n-1)$-element set (say $A$) and one $2$-element set (say $B$). Let $A=[n]\setminus \{i\}$, then any $2$-element set of the form $\{i,x\}$ different from $B$ will not be dominated either by $A$ or by $B$. Hence, $\gamma(G_{n-1,2}) > 2$. Therefore, $\gamma(G_{n-1,2}) =3$. Hence, the proof.	
	\end{proof}


\section*{Acknowlegdement}
We thank Shashanka Kulamarva, Research Scholar, Department of Mathematical and Computational Sciences, National Institute of Technology Karnataka, Surathkal for carefully reading our paper and giving valuable suggestions.

\end{document}